\documentclass{amsart}

\usepackage{lineno}
\usepackage[all]{xy}
\usepackage{amsmath}
\usepackage{amssymb}
\usepackage{array}
\usepackage{color}
\usepackage{verbatim}

\usepackage[
    colorlinks=true,
    linkcolor=blue,
    citecolor=blue,
    urlcolor=blue
]{hyperref}

\newtheorem{theorem}{Theorem}[section]
\newtheorem{lemma}[theorem]{Lemma}
\newtheorem{proposition}[theorem]{Proposition}

\theoremstyle{definition}

\theoremstyle{remark}
\newtheorem{remark}[theorem]{Remark}

\def\Q{{\mathbb Q}}
\def\Der{\mathrm{Der}}
\def\Hom{\mathrm{Hom}}

\def\cat0{\mathrm{cat}\_0}

\def\B{B\mathrm{aut}_1}

\def\im{\operatorname{Im}}

\numberwithin{equation}{section}

\begin{document}
\author{Yang Bai}
\address{School of Mathematical Sciences, Nankai University, Tianjin 300071, P.R.China}
\email{1687694204@qq.com}

\author{Xiugui Liu}
\address{School of Mathematical Sciences and LPMC, Nankai University, Tianjin 300071, P.R.China}
\email{xgliu@nankai.edu.cn}

\author{Jiaxi Zha}
\address{School of Mathematical Sciences, Nankai University, Tianjin 300071, P.R.China}
\email{1093913699@qq.com}

\thanks{The second author was supported in part by the National Natural Science Foundation of China (Grant No. 12171165).}
\subjclass[2020]{Primary 55P62; Secondary 55R35, 55P15}
\date{}
\title[Rational Realization of Classifying Spaces]{A complete classification of the rational realization of gap-two products of Eilenberg--MacLane spaces as classifying spaces}

\begin{abstract}
In this paper, we study the rational realization of the gap-two product $K(\Q,n)\times K(\Q,n+2)$ as the classifying space $\B(X)$ for a simply-connected $\pi$-finite space $X$ with rational homology of finite type. We prove that $K(\Q,n)\times K(\Q,n+2)$ can be realized as $\B(X)$ up to rational homotopy equivalence if and only if $n=3,4$ or $4m+2$ for $m\geq 1$. Moreover, we determine all their minimal Sullivan models up to isomorphism. The rational homotopy type of such $X$ is unique for $n=3$ or $4m+2$ for $m\geq 1$; there exist infinitely many rational homotopy types of such $X$ for $n=4$. 

\end{abstract}

\keywords{Rational homotopy theory, classifying spaces, Sullivan models, Lie models, derivations.}

\maketitle

\section{Introduction}

The classification theory for fibrations was developed in a series of foundational papers \cite{D,M,J}. For a fixed CW complex $X$, let ${\rm aut}(X)$ denote the monoid of self-homotopy equivalences of $X$. The fibrations with fibre homotopy equivalent to $X$ are classified by the Dold-Lashof classifying space $B{\rm aut}(X)$ \cite{DL}. In this paper, we consider $B{\rm aut}_1(X)$, the Dold-Lashof classifying space of the identity component of ${\rm aut}(X)$, which can be identified with the universal cover of $B{\rm aut}(X)$.

The classifying space $\B(X)$ was among the first objects described in rational homotopy theory, and the algebraic models for $B{\rm aut}_1(X)$ were given by many authors (cf. \cite[Ch.7]{T}). Sullivan gave a Lie model for $B{\rm aut}_1(X)$ in terms of the derivations of a minimal Sullivan model \cite{S}. Schlessinger and Stasheff constructed an equivalent model for $B{\rm aut}_1(X)$ based on the derivations of a minimal Quillen model \cite{SS}. 

A long-standing and open problem in rational homotopy theory was raised by Schlessinger in 1976 (see \cite[p. 519]{F-H-T}), which asks to characterize the space realized as the classifying space $\B (X)$ up to rational homotopy equivalence. Recall that a {\it $\pi$-finite space} is a space with only finitely many nonzero homotopy groups. In the context of simply-connected $\pi$-finite spaces with rational homology of finite type, we are concerned with the following two successive questions: whether a realization of a given space exists, and, if it does, which rational homotopy types realize the given space. 

Several results have been obtained in recent years assuming that $X$ is a simply-connected $\pi$-finite space. Lupton and Smith \cite{J-B, 2019L-S} constructed two infinite families of realizable spaces, and showed that $\mathbb{C}P^2$, $\mathbb{C}P^3$, $\mathbb{C}P^4$ and $S^4$ cannot occur as $\B(X)$ for a simply-connected $\pi$-finite space $X$. Bai, Liu and Xie \cite{BY} proved that the space with the minimal Sullivan model of the form $(\Lambda(a, b), db=a^p)$ for $|b|\leq 5n+2$ cannot be realized as $\B(X)$ for an $n$-connected $\pi$-finite space $X$. 

Other related work concerns even-dimensional spheres and gap-one products of Eilenberg--MacLane spaces. Lupton and Smith \cite{GE} studied the realization problem for $S^{2n}$ with $n=1,2,3$. Later, Bai and Liu \cite{BY1} extended their result for $S^{2n}$ to $n=4,5$, and proved the non-realization of gap-one products $ K(\Q^s,n) \times K(\Q^t,n+1)$ with $s,t\geq1$ in a finite range of degrees. The recent work of Bai, Liu and Zha \cite{BY2} removed these restrictions of degrees, which substantially generalized the previous results. 

Another aspect of the realization problem is to classify all the rational homotopy types realizing a prescribed space. This question was first considered by Yamaguchi in \cite{Y}, where all elliptic spaces $X$ for which $\B(X)$ has rank one were determined. Recently, Bai, Liu and Zha \cite{BY2} proved that, if $K(\Q^r,n)$ with $r\geq2$, $n\geq3$ can be realized as $\B (X)$ for $\pi$-finite space $X$, then $X$ has the rational homotopy type of $K(\Q^r,n-1)$. In addition to the results mentioned above, computations of $\B(X)$ for other classes of spaces $X$ can be found in \cite{JS1,JS2,JS3}.

In the present paper, we obtain a complete classification for the realization of gap-two products $K(\Q,n)\times K(\Q,n+2)$. More precisely, we determine exactly which degrees $n$ admit such realizations and classify, up to isomorphism, every realizing minimal Sullivan model. Our approach relies heavily on two structural propositions concerning the derivation Lie algebra, which are used to estimate the dimensions of the minimal Sullivan model; see Propositions \ref{prop} and \ref{sig}. In contrast to the classification of spaces realizing $K(\Q^r,n)$ for $r\geq2$, as well as the non-realization results for $S^{2n}$ and gap-one products, the gap-two case is substantially more involved, as illustrated by the following theorem.

\begin{theorem}\label{fi}
    Let $n$ be a positive integer. Then there exists a simply-connected $\pi$-finite space $X$ with rational homology of finite type such that  
    $$\B(X)\simeq_{\Q}K(\Q,n)\times K(\Q,n+2).$$ 
    if and only if 
    $$n\in \{3,4\} \cup \{ 4m+2: m\geq1 \} .$$ 
    Moreover, such $X$ has the minimal Sullivan model $M_X$ exactly as follows. 
    \begin{enumerate}
        \item When $n=3$, there is a unique model: 
        \begin{equation*}
            \begin{split}
                M_X\cong\ &\big( \Lambda(a_2,b_2,x_2,z_3,w_3,y_4) ,d \big), \\
                &da=db=dx=0,\ \ dz=ax, \\
                &dw=-bx, \ \ dy=aw+bz. 
            \end{split}
        \end{equation*}
        
        \item When $n=4$, there is an infinite family and an additional type: 
        \begin{equation*}
            \begin{split}
                M_X\cong\ &\big( \Lambda(x_2,u_{1,3},v_{1,3},\dots,u_{r,3},v_{r,3},y_5) ,d \big), \ \  \textit{   for }  r\geq0, \\ 
                &dx=du_{i}=dv_{i}=0, \ \ dy=x^3+\sum_{i=1}^r  u_iv_i,
            \end{split}
        \end{equation*}
        or 
        \begin{equation*}
            \begin{split}
                M_X\cong\ &\big( \Lambda(a_2,x_2,w_4,y_5) ,d \big), \\
                &da=dx=dw=0, \ \ dy=x^3+aw.
            \end{split}
        \end{equation*}

        \item When $n=4m+2$ for some $m\geq 1$, there is a unique model: 
        \begin{equation*}
            \begin{split}
                M_X\cong\ & \big( \Lambda(x_2,q_3,a_{2m+1},z_{2m+2},w_{2m+3},y_{4m+3}) ,d \big), \\
                &dx=da=0, \ \ dq=x^2, \ \ dz=xa, \\
                &dw=xz-qa,\ \ dy=z^2+2aw. 
            \end{split}
        \end{equation*}
    \end{enumerate}
\end{theorem}

\subsection*{Convention.} Unless otherwise specified, all spaces appearing in the rest of the paper are assumed to be simply-connected and to have rational homology of finite type. 

\subsection*{Organization.} In Section 2, we start by recalling basic facts about rational homotopy theory and classifying spaces. Then in Section 3, we reduce the realization problem to two cases according to dimensions. The subsequent discussion is divided into two branches: $n\geq5$ and $n\leq 4$. The branch $n\geq 5$ is considered in Section 4, where we introduce an auxiliary Sullivan algebra and establish two structural results for it. Finally in Section 5, we treat the branch $n\leq 4$ and complete the proof of Theorem \ref{fi}. 

\subsection*{Statement on AI use}{The authors used GPT--5.6 Sol to assist with computations in several special cases ($n=5, 6, 7, 8$). Based on these computations, the authors developed the general methods presented in this paper. The entire manuscript was written by the authors, who independently verified all computations and mathematical arguments, and take full responsibility for the correctness of its content.}

\section{Preliminaries}

In this section, we first recall basic facts in rational homotopy theory and the derivation models for classifying spaces. Standard references are \cite{F-H,F-H-T}.

A {\it Sullivan algebra} is a commutative differential graded algebra (CDGA for short) of the form $(\Lambda V,d )$, where $V=V^{\geq 1}$ and $d$ satisfies the nilpotence condition: there exists an increasing filtration $\{V(k)\}_{k\geq 0}$ such that $V=\cup_{k}V(k)$ and $d(V(k))\subset \Lambda V(k-1)$. A Sullivan algebra $(\Lambda V,d )$ is called {\it minimal} if the differential $d$ is decomposable, i.e., $d(V)\subset \Lambda^+ V\cdot \Lambda^+ V $. Here $\Lambda^+ V$ is the augmentation ideal of $\Lambda V$. A (minimal) Sullivan algebra $(\Lambda V,d )$ is called a {\it (minimal) Sullivan model for $X$} if there exists a quasi-isomorphism $ (\Lambda V,d ) \xrightarrow{\simeq} A_{PL}(X) $. 

For a Sullivan algebra $(\Lambda V,d )$, let $ f:\Lambda V\longrightarrow \Lambda V $ be an isomorphism of algebras. Then we can obtain a new Sullivan algebra $(\Lambda V, fdf^{-1} )$, and $f$ can be considered as an isomorphism of CDGAs
$$f:(\Lambda V,d)\longrightarrow (\Lambda V, fdf^{-1}).$$
This process is called a {\it change of basis}. 

The degree of an element $x$ is denoted by $|x|$. A {\it derivation $\theta$  that reduces degree by $n$} is a linear map $$\theta: \Lambda V\rightarrow \Lambda V$$ of degree $-n$ such that $$\theta(xy)=\theta(x)y+(-1)^{n|x|}x\theta(y)$$ for $x, y\in \Lambda V$. Let $({\rm Der}(\Lambda V),D)$ denote the differential graded Lie algebra (abbr. DGL) of derivations of $\Lambda V$, where the Lie bracket is given by $$[\theta_1,\theta_2]=\theta_1\theta_2-(-1)^{|\theta_1||\theta_2|}\theta_2\theta_1,\ \ \theta_1,\theta_2 \in {\rm Der}(\Lambda V)$$
and the differential $D$ is given by 
$$D(\theta)=[d,\theta], \ \ \theta\in {\rm Der}(\Lambda V). $$ 
For convenience, we denote by ${\rm Der}_k(\Lambda V)$ the set of derivations that reduce degree by $k$. In particular, we remark that $(\Der(\Lambda V),D)$ can be viewed as a right $(\Lambda V,d)$-module by setting
$$\big( \theta \cdot x\big)(v)=(-1)^{|x||v|}\theta(v)x$$
for $\theta\in \Der(\Lambda V)$ and $x,v\in \Lambda V$, where the module action is given by
\begin{align*}
    \begin{split}
        \Der_k(\Lambda V) \times(\Lambda V)^n &\rightarrow \Der_{k-n}(\Lambda V) \\
        (\theta,x)&\mapsto \theta \cdot x.
    \end{split}
\end{align*}
It is worth emphasizing that $\theta\cdot x$ or simply $\theta x$ denotes the action of $x$ on the derivation $\theta$, whereas $\theta(x)$ denotes the image that the derivation $\theta$ sends $x$ to. 

In the following decomposition
$$\Lambda V=\bigoplus_{l\geq0}\Lambda^l V$$
with $\Lambda^0 V=\Q$, the elements in $\Lambda^l V$ have wordlength $l$. For any $x\in \Lambda^+ V$, $x$ can be decomposed as
$$x=x_1+x_2+\dots,$$
where $x_l\in \Lambda^l V$ for $l\geq 1$. We denote $x_l$ by $L_{l}(x)$ in the above decomposition of $x$. Similarly, by the isomorphism $\Der (\Lambda V)\cong \Hom (V,\Lambda V)$, wordlength in $\Lambda V$ also induces a grading on $\Der (\Lambda V)$, i.e.,
$$\Der(\Lambda V)\cong \bigoplus_{l\geq -1}\Hom (V,\Lambda ^{l+1}V). $$
Thus for any $\theta\in \Der(\Lambda V)$, $\theta$ can be decomposed as
$$\theta=\theta_{-1}+\theta_0+\dots,$$
where $\theta_l\in \Hom (V,\Lambda ^{l+1}V)$ is a derivation that increases wordlength by $l$. We denote $\theta_l$ by $L_{l}(\theta)$ in the decomposition of $\theta$. Note that the differential $d$ in $\Lambda V$ decomposes as 
$$d=d_1+d_2+\dots,$$
where $d_l=L_l(d)$. Therefore the differential $D=[d,-]$ in $\Der(\Lambda V)$ increases wordlength by at least by 1.

Let ${\rm Der}'(\Lambda V)$ denote the Lie subalgebra of ${\rm Der}(\Lambda V)$ defined by
$${\rm Der}'_k(\Lambda V)=\left\{
\begin{array}{ll}
{\rm Der}_k(\Lambda V),&k>1,\\
Z({\rm Der}_1(\Lambda V)),&k=1,\\
0,&k<1,
\end{array}
\right.$$
where $Z({\rm Der}_1(\Lambda V))$ is the subspace of all cycles in ${\rm Der}_1(\Lambda V)$. We recall Sullivan's original result about the model for the classifying space $B{\rm aut}_1(X)$.

\begin{theorem}\label{liu}\rm{(\cite[p. 313]{S})}
Let $X$ be a space with minimal Sullivan model $(\Lambda V, d)$. Then the DGL $({\rm Der}'(\Lambda V),D)$ is a Lie model for $B{\rm aut}_1(X)$.
\end{theorem}

Recall that the homotopy groups of the loop space $\Omega B{\rm aut}_1(X)$ admit a natural bilinear pairing $[-,-]$ called the \textit{Samelson product} (see \cite[Ch. \uppercase\expandafter{\romannumeral3}]{W}). In particular, we have the following theorem.

\begin{theorem}\label{J}
Let $X$ be a space
with minimal Sullivan model $(\Lambda V, d)$. Then there is an isomorphism of graded Lie algebras
$$\pi_{\ast}(\Omega B{\rm aut}_1(X)) \cong H_{\ast}({\rm Der}'(\Lambda V),D)$$
where the left-hand graded space has the Samelson product. 
\end{theorem}

In the previous work of Lupton and Smith \cite{J-B, 2019L-S}, they studied the realization problem for $\pi$-finite spaces, relying heavily on the following theorem.

\begin{theorem}\label{pf}\rm{(\cite[Prop. 2.2]{J-B})}
Suppose $X$ is $\pi$-finite with
$$\pi_i(X) =\left\{
\begin{array}{ll}
\mathbb{Q}^r\ {\rm some} \ r \geq 1, & i = N,\\
0, &i > N.
\end{array}
\right.$$
Then we have
$$\pi_i(B{\rm aut}_{1}(X)) =\left\{
\begin{array}{ll}
\mathbb{Q}^r,& i = N+1,\\
0, &i > N+1.
\end{array}
\right.$$
\end{theorem}

\section{Reduction to two cases}

Let $(\Lambda V,d)$ be a minimal Sullivan model for $X$ with $V=V^{\geq2}$ of finite type, and let $\B(X)\simeq_{\Q}K(\Q,n)\times K(\Q,n+2)$. Recall that we have the isomorphism $\Der (\Lambda V)\cong \Hom (V,\Lambda V)$. In the rest of the paper, for convenience, we denote by $(v,x)$ the derivation that sends $v$ to $x$ and other generators to zero, and denote by $v^\ast$ the derivation that sends $v$ to $1$ and other generators to zero, where $v\in V$ and $x\in \Lambda V$.

In this section, we reduce the realization problem for $K(\Q,n)\times K(\Q,n+2)$ to two cases. For this, take $N=n+1$. By Theorem \ref{J}, we have 
\begin{equation}\label{td0}
 H_k({\rm Der}'(\Lambda V),D)=\left\{
\begin{array}{ll}
\Q, &k=N,\\
\Q, &k=N-2,\\
0, &{\rm otherwise}.
\end{array}
\right.
\end{equation}
Since $X$ is $\pi$-finite by assumption, Theorem \ref{pf} yields that 
$$V^i =\left\{
\begin{array}{ll}
\mathbb{Q} , & i = N,\\
0, &i > N.
\end{array}
\right.$$
Set $V^N=\Q y$. In the following sequence
$$\Der_N(\Lambda V)\xrightarrow{D_N=0} \Der_{N-1}(\Lambda V) \xrightarrow{\ D_{N-1}\ }\Der_{N-2}(\Lambda V), $$
we have 
\begin{align*}
    \begin{split}
        \Der_N(\Lambda V) &= \Q y^\ast ,\\
        \Der_{N-1}(\Lambda V) &= (V^{N-1})^\ast ,\\
        \Der_{N-2}(\Lambda V) &= (V^{N-2})^\ast \oplus (y,V^2) , 
    \end{split}
\end{align*}
for degree reasons. The vanishing of homology in degree $N-1$ implies that $D_{N-1}$ is injective. Define the subspace 
$$A=\big( D(V^{N-1})^\ast  \big)(y) \ \subset V^2. $$
Then we have $D(V^{N-1})^\ast =(y,A)$. The injectivity of $D_{N-1}$ gives $\dim V^{N-1}=\dim A$. 

Recall that $L_l(\theta)$ denotes the summand of $\theta\in \Der(\Lambda V)$ which increases wordlength by $l$. For convenience, we call $L_{-1}(\theta)$ the constant part of $\theta$, and call $L_{0}(\theta)$ the linear part of $\theta$. The following lemma follows immediately from our earlier work \cite[Proposition 3.3]{BY2}. 
\begin{lemma}\label{codim}
    $A$ has codimension $1$ in $V^2$. 
\end{lemma}
\begin{proof}
    Note that 
    $$\Q\cong H_{N-2}(\Der^\prime(\Lambda V))\cong \ker D|_{(V^{N-2})^{\ast}}\oplus \frac{(y,V^2)}{(y,A)}. $$
    Suppose that $H_{N-2}(\Der^\prime(\Lambda V))$ is generated by a derivation with nonzero constant part. By \cite[Proposition 3.3]{BY2}, this leads to a contradiction. Thus we have 
    $$ \Q\cong H_{N-2}(\Der^\prime(\Lambda V))\cong\frac{(y,V^2)}{(y,A)} \cong V^2/A , $$
    which implies $\dim (V^2/A)=1$. 
    
\end{proof}

We can write $V^2=A\oplus \Q x$ for some $x\in V^2$, and we have 
$$H_{N}(\Der^\prime(\Lambda V))\cong \Q \cdot y^\ast ,\ \  H_{N-2}(\Der^\prime(\Lambda V))\cong \Q \cdot(y,x). $$
Note that 
$$V^2A\oplus \Q x^2=\Lambda^2V^2,\ \ V^2A+dV^3\subset \Lambda^2V^2,$$
and we set 
$$\delta=\dim (\Lambda^2V^2) -\dim (V^2A+dV^3)\in \{0,1\}.$$
Therefore, the subsequent discussion can be divided into the following two cases. 

\noindent {\bf Case 1:} $\delta=0$. We set $V^3=B\oplus \Q q$, where 
$$dB\subset V^2A,\ \ dq=Q \equiv x^2 \ ({\rm mod }\  V^2A).$$

\noindent {\bf Case 2:} $\delta=1$. We have 
$$dV^3\subset V^2A.$$

For the first case $\delta=0$, since $dq =Q\equiv x^2 \ ({\rm mod }\  V^2A)$, we can set 
$$dq=Q=x^2+\sum v_i a_i,$$
where $\{a_i\}$ is a basis for $A$ and $v_i\in V^2$.

\section{The branch $n\geq 5$}

In this section, we treat the branch $n\geq5$. First in Subsection \ref{sub1}, we introduce an auxiliary Sullivan algebra. Then in Subsection \ref{sub2}, we rule out the case $\delta=1$. Subsequently in Subsections \ref{sub3} and \ref{sub4}, we further consider the case $\delta=0$. Finally in Subsection \ref{sub5}, we present all realizable cases.

\subsection{The auxiliary Sullivan algebra}\label{sub1}

\ 

We have $N=n+1\geq 6$. For both cases $\delta=0,1$, we shall construct an auxiliary Sullivan algebra 
$$(\Lambda U\otimes \Lambda V^{<N} ,\partial ),$$
where $\partial|_{V^{<N}}=d|_{V^{<N}}$, and $U$ will be defined for both cases separately. 

For the case $\delta=0$, define the graded vector space 
\begin{equation}\label{u0}
U^k=\left\{
\begin{array}{ll}
\overline{A}, &k=1,\\
\overline{B}, &k=2,\\
\overline{V^{k+1}}, &k=3,\dots, N-2, 
\end{array}
\right.
\end{equation}
where the overline denotes the suspension of graded vector spaces. In what follows, we shall inductively define the differential $\partial$ on $U$. First, for the basis $\{a_i\}$ of $A$, we define 
$$d_1 \overline{a_i}=a_i,\ \ d_1|_{V^{<N}}=d|_{V^{<N}},$$
which gives a differential $d_1$ on $\Lambda U^1\otimes \Lambda V^{<N}$. There exists an algebra isomorphism $f_2: \Lambda U\otimes \Lambda V^{<N} \rightarrow \Lambda U\otimes \Lambda V^{<N}$ with 
\begin{equation*}
    \begin{split}
        &f_2(\overline{a_i})=\overline{a_i}, \ \ f_2(a_i)=a_i, \ \ f_2(x)=x, \\
        &f_2(b_j)=b_j-\sum v_{j,i} \overline{a_i},\ \ f_2(q)=q-\sum v_{i} \overline{a_i},
    \end{split}
\end{equation*}
and $f_2$ restricts to the identity on generators of higher degrees, where $\{b_j\}$ is a basis for $B$ and $db_j=\sum v_{j,i}a_i$. It follows that the restriction 
$$f_2|_{\Lambda U^1\otimes \Lambda V^{<N} }:(\Lambda U^1\otimes \Lambda V^{<N} ,f_2^{-1} d_1 f_2)\rightarrow (\Lambda U^1\otimes \Lambda V^{<N}, d_1)$$
is an isomorphism of CDGAs. We denote the differential $f_2^{-1} d_1 f_2$ by $d_2$, and $d_2$ satisfies 
$$d_2(\overline{a_i})=a_i, \ \ d_2(b_j)=0,\ \ d_2(q)=x^2.$$
Therefore, we can extend the differential $d_2$ to $U^2$ by setting $d_2(\overline{b_j})=b_j$. 

Suppose that we have defined the differential $d_{k-1}$ on $\Lambda U^{< k}\otimes \Lambda V^{<N}$ with
\begin{equation*}
    \begin{split}
        d_{k-1}x=0,\ \ \ d_{k-1}(q)=x^2,\ \ \ d_{k-1}(\overline{v})=v, \ \ \overline{v}\in U^{<k}\subset \overline{V^{\leq k}}
    \end{split}
\end{equation*}
for some $3\leq k\leq N-2$. For a basis $\{w_s\}$ of $ V^{k+1}$, note that $d_{k-1}(w_s)$ is a cycle in $(\Lambda U^{< k}\otimes \Lambda V^{\leq k},d_{k-1})$, which is isomorphic to 
$$E(U^{< k})\otimes \big(\Lambda(x,q),dq=x^2\big),$$
where $E(U^{< k})$ is an acyclic Sullivan algebra. Thus $d_{k-1}(w_s)$ is an exact cycle. Put 
$$d_{k-1}(w_s)=d_{k-1}(p_s)$$
for some $p_s\in \Lambda U^{<k}\otimes \Lambda V^{\leq k}$. Then there exists an algebra isomorphism $f_{k}: \Lambda U\otimes \Lambda V^{<N} \rightarrow \Lambda U\otimes \Lambda V^{<N}$ with 
\begin{equation*}
    f_k(w_s)=w_s-p_s,
\end{equation*}
and $f_k$ restricts to the identity on other generators. Similarly, the extended differential $d_k$ on $\Lambda U^{\leq k}\otimes \Lambda V^{<N }$ is given by $f_k^{-1}d_{k-1}f_k$ and $d_k(\overline{w_s})=w_s$. 

By induction, we obtain a differential $d_{N-2}$ on $\Lambda U\otimes \Lambda V^{<N} $, together with a sequence of algebra isomorphisms 
$$\Lambda U\otimes \Lambda V^{<N} \xrightarrow{f_{N-2}} \Lambda U\otimes \Lambda V^{<N}\xrightarrow{f_{N-3}}\dots\xrightarrow{\ f_2\ } \Lambda U\otimes \Lambda V^{<N} .$$
Finally, we denote 
$$\partial=(f_2\dots f_{N-2}) d_{N-2}(f_2\dots f_{N-2})^{-1}.$$

For the other case $\delta=1$, we first add a generator $t$ with $dt=x^2$. Define 
\begin{equation}\label{u1}
 U^k=\left\{
\begin{array}{ll}
\overline{A}, &k=1,\\
\overline{V^3}, &k=2,\\
\overline{V^4}\oplus\Q t, &k=3,\\
\overline{V^{k+1}}, &k=4,\dots, N-2. 
\end{array}
\right.
\end{equation}
Then by similar arguments as above, we can define the differential $\partial$ on $\Lambda U\otimes \Lambda V^{<N}$. 

For convenience, elements in $U$ are called {\it formal }, and elements in $V$ are called {\it actual}. The auxiliary Sullivan algebra $(\Lambda U\otimes \Lambda V^{<N},\partial)$ satisfies the following properties. 
\begin{lemma}\label{alg}
    Denote by $\partial_1$ the part of $\partial $ that increases wordlength of actual elements by exactly one. 
    \begin{enumerate}
        \item In both cases, we have $\partial|_{V^{<N}}=d|_{V^{<N}}$ and $\partial_1^2=0$.

        \item In the case $\delta=0$, we have 
        $$H(\Lambda U\otimes \Lambda V^{<N},\partial)=\Q\oplus \Q x, \ \ H(\Lambda U\otimes \Lambda V^{<N},\partial_1)=\Q\oplus \Q x.$$
        In particular, every pure formal $\partial_1$-cycle of positive degree is zero. 

        \item In the case $\delta=1$, we have 
        $$H(\Lambda U\otimes \Lambda V^{<N},\partial)=\Q\oplus \Q x, \ \ H(\Lambda U\otimes \Lambda V^{<N},\partial_1)=\Lambda(x,t).$$
        In particular, all pure formal $\partial_1$-cycles of positive degree are in $\Q t$. 
    \end{enumerate}
\end{lemma}

\begin{proof}
    (1) Recall that in the sequence of algebra isomorphisms 
    $$\Lambda U\otimes \Lambda V^{<N} \xrightarrow{f_{N-2}} \Lambda U\otimes \Lambda V^{<N}\xrightarrow{f_{N-3}}\dots\xrightarrow{\ f_2\ } \Lambda U\otimes \Lambda V^{<N} ,$$
    $f_kd_kf_k^{-1}$ extends $d_{k-1}$ by definition. Thus 
    $$\partial=f_2\dots f_{N-2} d_{N-2} f_{N-2}^{-1}\dots f_2^{-1}$$
    extends $d_1$ on $V^{<N}$, which is exactly $d|_{V^{<N}}$. 

    According to the definition, for $\overline{v}\in U^k$, we have $d_{k-1}(v)=d_{k-1}(p_v)$ for some $p_v\in \Lambda U^{<k}\otimes \Lambda V^{\leq k}$. Then we deduce that 
    \begin{equation*}
        \begin{split}
            \partial(\overline{v})&=f_2\dots f_{N-2} d_{N-2} f_{N-2}^{-1}\dots f_2^{-1} (\overline{v}) \\
            &=f_2\dots f_{k-1} f_{k}d_{k}(\overline{v}) \\
            &=f_2\dots f_{k-1} f_k (v) \\
            &=f_2\dots f_{k-1} (v-p_v) \\
            &=v- p_v^\prime, 
        \end{split}
    \end{equation*}
    where $p_v^\prime=f_2\dots f_{k-1} (p_v)$ is decomposable. Moreover, we have 
    $$d(v)=\partial(v)=\partial(p_v^\prime),$$
    which implies $p_v^\prime$ has wordlength of actual elements at least one. Thus $\partial $ increases wordlength of actual elements by at least one. It follows that $\partial_1^2=0$ by wordlength arguments. 

    (2) By \cite[Proposition 14.13]{F-H-T}, we can show the minimal Sullivan models for $(\Lambda U\otimes \Lambda V^{<N},\partial)$ and $(\Lambda U\otimes \Lambda V^{<N},\partial_1)$ are both $(\Lambda(x,q),dq=x^2)$. By direct computation, we have 
    $$H(\Lambda U\otimes \Lambda V^{<N},\partial)=\Q\oplus \Q x, \ \ H(\Lambda U\otimes \Lambda V^{<N},\partial_1)=\Q\oplus \Q x.$$

    (3) Note that $\partial(t)=x^2$ increases wordlength of actual elements by two in the case $\delta=1$, and thus $\partial_1(t)=0$. Similarly, by \cite[Proposition 14.13]{F-H-T}, we can show the minimal Sullivan model for $(\Lambda U\otimes \Lambda V^{<N},\partial)$ is $(\Lambda(x,t),dt=x^2)$, and the minimal Sullivan model for $(\Lambda U\otimes \Lambda V^{<N},\partial_1)$ is $(\Lambda(x,t),dx=dt=0)$. By direct computation, we have 
    $$H(\Lambda U\otimes \Lambda V^{<N},\partial)=\Q\oplus \Q x, \ \ H(\Lambda U\otimes \Lambda V^{<N},\partial_1)=\Lambda(x,t).$$

\end{proof}

\subsection{Injectivity of the constant part}\label{sub2}

\

Recall that 
\begin{equation}\label{td}
H_k({\rm Der}'(\Lambda V),D)=\left\{
\begin{array}{ll}
\Q y^\ast, &k=N,\\
\Q (y,x), &k=N-2,\\
0, &{\rm otherwise}.
\end{array}
\right.
\end{equation}
We first introduce a lifting of the exact derivation $(y,x^2)$ in the case $\delta=1$. 

\begin{lemma}\label{xx}
    If $\delta=1$, there exists a derivation $\eta$ of degree $N-3$ such that 
    $$D\eta=(-1)^N(y,x^2), \ \ L_{-1}(\eta)\neq 0. $$
\end{lemma}

\begin{proof}
    The derivation $(y,x^2)$ is a cycle by direct computation. It is exact by (\ref{td}), which gives the derivation $\eta$ with $D\eta=(-1)^N(y,x^2)$. 

    Suppose that $L_{-1}(\eta)= 0$. For degree reasons, we have 
    $$\eta \in (V^{N-1},V^2)\oplus (y,V^3).$$
    Thus $D\eta= (-1)^N(y,x^2)\in (y,V^2A)+(y,dV^3)$, which contradicts the fact that 
    $$dV^3\subset V^2A,\ \ V^2A\oplus \Q x^2=\Lambda^2V^2.$$

\end{proof}

\begin{proposition}\label{prop}
    With notation as above. For the auxiliary Sullivan algebra $(\Lambda U\otimes \Lambda V^{<N},\partial)$, there is a right $(\Lambda V^{<N},d)$-module map 
    $$\Phi: (\Lambda U)^{\leq N-2}\otimes \Lambda V^{<N} \longrightarrow \Der(\Lambda V)$$
    sending an element of degree $k$ to a derivation of degree $N-k$, which satisfies 
    $$\Phi(1)= y^\ast,\ \ D\Phi=(-1)^N\Phi \partial,$$
    and, in the case $\delta=1$, $\Phi(t)=\eta$. Moreover, the restriction map 
    $$L_{-1}\Phi: (\Lambda U)^{\leq N-2}\longrightarrow (V)^\ast$$
    is injective. 

\end{proposition}

\begin{proof}
    In order to construct the right $\Lambda V^{<N}$-module map $\Phi$, it suffices to define $\Phi$ on $(\Lambda U)^{\leq N-2}$. For the basis $\{a_i\}$ of $A$, the derivation $(-1)^N(y,a_i)$ is exact by (\ref{td}), and thus we can set 
    $$D\theta_i=(-1)^N(y,a_i). $$
    First, we define $\Phi$ on $(\Lambda U)^{\leq 1}=\Q \oplus\overline{A}$ by setting 
    $$\Phi(1)=y^\ast,\ \ \Phi(\overline{a_i})=\theta_i.$$ 
    
    Suppose that we have defined $\Phi$ on $(\Lambda U)^{<k}$ for some $2\leq k\leq N-2$. Recall that $\partial$ increases wordlength of actual elements by at least one, and actual elements have degree at least two. For a basis element $u\in (\Lambda U)^k$, we have
    $$\partial(u)\in (\Lambda U)^{<k}\otimes \Lambda V^{<N},$$
    on which $\Phi$ has been defined. Thus, $(-1)^N\Phi\partial(u)$ is a cycle derivation, which is exact by (\ref{td}) and $|\partial(u)|=k+1\geq 3$. Set 
    $$(-1)^N\Phi\partial(u)=D(\xi_u).$$
    In particular, we require $\xi_u=\eta$ for $u=t$ in the case $\delta=1$. Then, we define $\Phi$ on $(\Lambda U)^k$ by setting 
    $$\Phi(u)=\xi_u.$$ 
    By induction, this completes the definition of $\Phi$ on $(\Lambda U)^{\leq N-2}\otimes \Lambda V^{<N}$. According to the construction, we have $D\Phi=(-1)^N\Phi \partial$. 

    Finally, we show the injectivity of $L_{-1}\Phi|_{(\Lambda U)^{\leq N-2}}$ by induction on degrees. Since $\Phi(1)=y\ast$ and $L_{-1}(y^\ast)\neq 0$, $L_{-1}\Phi$ is injective on $(\Lambda U)^0$. Suppose we have shown that $L_{-1}\Phi$ is injective on $(\Lambda U)^{<k}$ for some $1\leq k\leq N-2$. By wordlength arguments, in order to prove the injectivity of $L_{-1}\Phi|_{(\Lambda U)^{k}}$, it suffices to show that $L_0D\Phi|_{(\Lambda U)^{k}}$ is injective. Since 
    $$\im(\partial-\partial_1)\subset  (\Lambda U)^{<k}\otimes \Lambda^{\geq 2} V^{<N},\ \ \ L_0\Phi\big( (\Lambda U)^{<k}\otimes \Lambda^{\geq 2} V^{<N}  \big)=0,$$
    we can derive that 
    $$(-1)^NL_0D\Phi=L_0 \Phi \partial =L_0 \Phi \partial_1.$$
    Recall that $\partial_1$ increases wordlength of actual elements by exactly one, and actual elements have degree at least two. Thus in the restriction of the composition map 
    $$(\Lambda U)^k\xrightarrow{\ \partial_1 \ } (\Lambda U)^{<k}\otimes V^{<N} \xrightarrow{\  L_0\Phi \ } \Hom(V,V),$$
    the map $L_0\Phi:(\Lambda U)^{<k}\otimes V^{<N}\rightarrow \Hom(V,V)$ can be identified with 
    $$(L_{-1}\Phi)\otimes {\rm id}:(\Lambda U)^{<k}\otimes V^{<N}\rightarrow (V)^\ast \otimes V^{<N} $$
    up to sign, which is injective by the induction hypothesis. Moreover, $\partial_1|_{(\Lambda U)^k}$ is also injective except on $\Q t$ by Lemma \ref{alg}. Hence we need to check the case when $\delta=1$ and $k=3$ separately. 
    
    In that case, by (\ref{u1}) we deduce that 
    $$(\Lambda U)^k=(\Lambda U)^3=\Lambda^3\overline{A}\oplus \overline{V^3} \overline{A} \oplus \overline{V^4}\oplus \Q t. $$
    Suppose that $L_{-1}\Phi (u_1+u_2+u_3+\lambda t)=0$ with 
    $$u_1\in \Lambda^3\overline{A}, \ \ u_2\in \overline{V^3} \overline{A}, \ \ u_3\in   \overline{V^4}, \ \ \lambda\in \Q.$$
    By similar arguments as above, it follows that $L_{0}\Phi\partial_1(u_1+u_2+u_3)=0$, which implies $u_1=u_2=u_3=0$ and $\lambda L_{-1}\Phi ( t)=0$. By construction, $\Phi(t)=\eta$. By Lemma \ref{xx}, $L_{-1}(\eta)\neq 0$, which implies $\lambda=0$. This completes the proof of injectivity of $L_{-1}\Phi|_{(\Lambda U)^{k}}$. By induction, we finally reach the conclusion that $L_{-1}\Phi: (\Lambda U)^{\leq N-2}\rightarrow (V)^\ast$ is injective.

\end{proof}

Proposition \ref{prop} imposes strong restrictions on the original minimal Sullivan model $(\Lambda V,d)$, and a simple application can rule out the case $\delta=1$, as demonstrated below. 

\begin{lemma}\label{pc}
    The case $\delta=1$ is impossible for $N\geq 6$. 
\end{lemma}

\begin{proof}
    In the case $\delta=1$, using the injectivity in Proposition \ref{prop}, we have 
    $$\dim V^4+1\leq \dim (\Lambda U)^3\leq \dim(V^{N-3})^\ast=\dim V^{N-3},$$
    On the other hand, since $N-4\geq 2$, it follows that $\dim V^{N-3}\leq \dim (\Lambda U)^{N-4}$ by (\ref{u1}). Thus we also have 
    $$\dim V^{N-3}\leq \dim (\Lambda U)^{N-4}\leq \dim(V^{4})^\ast =\dim V^4.$$
    These two inequalities lead to a contradiction.

\end{proof}

\subsection{Dimension estimates for the case $\delta=0$}\label{sub3}

\ 

In order to handle the only possible case $\delta=0$ for $N=n+1\geq 6$, we first introduce the following map 
\begin{equation*}
    \begin{split}
        \Sigma:(V^{<N})^\ast&\longrightarrow V^{<N}  \\
        \theta  &\longrightarrow \big(L_{0}(D\theta)\big)(y),
    \end{split}
\end{equation*}
which sends a derivation of degree $k$ to an element of degree $N+1-k$. 

\begin{proposition}\label{sig}
    For the map $\Sigma $ defined as above, we have 
    \begin{enumerate}
        \item $\theta_1 (\Sigma\theta_2)=(-1)^{|\theta_1|+|\theta_2|+|\theta_1||\theta_2|} \theta_2 (\Sigma\theta_1)$ for $\theta_1,\theta_2\in (V^{<N})^\ast$; 
        
        \item $\Sigma L_{-1}\Phi|_{U}$ coincides with the desuspension of $U$ up to sign; 

        \item $\im(\Sigma)=\widetilde{U}$ and $\ker(\Sigma)= \Q x^\ast \oplus \Q q^\ast $, where $\widetilde{U}\subset V^{<N}$ denotes the desuspension of $U$. 
    \end{enumerate}
    
\end{proposition}

\begin{proof}
    (1) Since $\theta_1(y)=\theta_2(y)=0$ and $[\theta_1,\theta_2]=0$, by direct computation we have 
    \begin{equation*}
        \begin{split}
            0&=\big(D[\theta_1,\theta_2]\big)(y) \\
            &= \big([D\theta_1,\theta_2]+(-1)^{|\theta_1|} [\theta_1,D\theta_2]  \big)(y) \\
            &= -(-1)^{(|\theta_1|-1)|\theta_2|} \theta_2\big( (D\theta_1)(y) \big) +(-1)^{|\theta_1|} \theta_1\big( (D\theta_2)(y) \big) 
        \end{split}
    \end{equation*}
    By considering the linear part of $D\theta_1$ and $D\theta_2$, we conclude that 
    \begin{equation}\label{do}
        \theta_1 (\Sigma\theta_2)=(-1)^{|\theta_1|+|\theta_2|+|\theta_1||\theta_2|} \theta_2 (\Sigma\theta_1).
    \end{equation}

    (2) For $u\in U$, denote by $\widetilde{u}$ its desuspension. Note that 
    $$\widetilde{u} \in V^{<N} ,\ \ \partial u=\widetilde{u}-p^\prime_{\widetilde{u}}$$
    By direct computation we have 
    \begin{equation*}
        \begin{split}
            \Sigma L_{-1}\Phi(u)&= \big(L_{0} D\Phi(u)\big)(y)  \\
            & =(-1)^N \big(L_{0} \Phi(\partial u)\big)(y) \\
            &=(-1)^N \big(L_{0} \Phi(\widetilde{u}-p^\prime_{\widetilde{u}})\big)(y) . 
        \end{split}
    \end{equation*}
    Note that $p^\prime_{\widetilde{u}}\in \Lambda^{+}(U\oplus V^{<N}) \otimes V^{<N}$, which implies $\big(L_{0} \Phi(p^\prime_{\widetilde{u}})\big)(y)=0$. Thus we conclude that 
    \begin{equation*}
        \begin{split}
            \Sigma L_{-1}\Phi(u)&=(-1)^N\big(L_{0} \Phi(\widetilde{u})\big)(y) \\
            &=(-1)^N\big(y^\ast \cdot \widetilde{u}\big)(y) \\
            &=(-1)^{N+|y||\widetilde{u}|} \widetilde{u},
        \end{split}
    \end{equation*}
    which implies that $\Sigma L_{-1}\Phi|_{U}$ coincides with the desuspension of $U$ up to sign. 

    (3) For a fixed derivation $\theta_1\in (V^k)^\ast$, (\ref{do}) implies 
    \begin{equation*}
        \begin{split}
            \theta_1 \in \ker\Sigma  &\ \iff  \ \theta_2(\Sigma\theta_1)=0, \ \forall \theta_2\in (V^{N+1-k})^\ast  \\
            &\ \iff \  \theta_1(\Sigma\theta_2)=0, \ \forall \theta_2\in (V^{N+1-k})^\ast \\
            &\ \iff \ \theta_1(\im(\Sigma))=0. 
        \end{split}
    \end{equation*}
    Since 
    \begin{equation*}
        \widetilde{U}^k=\left\{
        \begin{array}{ll}
        A, &k=2,\\
        B, &k=3,\\
        V^{k}, &k=4,\dots, N-1, 
        \end{array}
        \right.
    \end{equation*}
    it suffices to prove $\im(\Sigma)=\widetilde{U}$. 
    
    According to the fact that $\Sigma L_{-1}\Phi|_{U}$ coincides with the desuspension of $U$ up to sign, we have $\widetilde{U}\subset \im(\Sigma)$. Thus we only need to check 
    $$\Sigma \big((V^{N-1})^\ast\big)\subset  A,\ \ \Sigma \big((V^{N-2})^\ast\big)\subset B.$$
    We note that 
    \begin{align*}
    \begin{split}
        \Der_{N-1}(\Lambda V) &= (V^{N-1})^\ast, \\
        \Der_{N-2}(\Lambda V) &= (V^{N-2})^\ast \oplus (y,A)\oplus \Q (y,x) ,\\
        \Der_{N-3}(\Lambda V) &= (V^{N-3})^\ast \oplus (y,B)\oplus \Q(y,q) \oplus (V^{N-1},V^2),
    \end{split}
    \end{align*}
    and the differential satisfies
    \begin{align*}
    \begin{split}
        &D_{N-1}: (V^{N-1})^\ast \xrightarrow{\cong} (y,A), \\
        &D_{N-2}: (V^{N-2})^\ast \xrightarrow{\cong} Z\big((y,B)\oplus \Q(y,q)\oplus (V^{N-1},V^2)\big), 
    \end{split}
    \end{align*}
    where $Z$ denotes the cycles in a complex. The fact that $D_{N-1}\big((V^{N-1})^\ast \big)= (y,A)$ implies $\Sigma \big((V^{N-1})^\ast\big)\subset  A$.  According to the fact that $V^3=B\oplus \Q q$ with $dB\subset V^2A$ and $dq \equiv x^2 \ ({\rm mod }\  V^2A)$, we have 
    $$D_{N-3}\big((y,B)\oplus (V^{N-1},V^2)\big)= (y, V^2 A),$$
    $$D_{N-3}(y,q)= (y,x^2) \ ({\rm mod }\  (y,V^2A)), $$
    which implies that no cycle derivation has a nonzero $(y,q)$-part. Therefore we conclude that 
    $$D_{N-2}\big( (V^{N-2})^\ast \big)\subset (y,B)\oplus (V^{N-1},V^2),$$
    which implies $\Sigma \big((V^{N-2})^\ast\big)\subset B$.

\end{proof}

Now we are in a position to estimate the dimensions of $(\Lambda U)^{\leq N-2}\otimes \Lambda V^{<N}$. For convenience, we set 
\begin{equation*}
    \begin{split}
        \alpha_k&=\dim U^k=\dim \widetilde{U}^{k+1}, \\
        \beta_k&=\dim (\Lambda^{\geq 2}U)^k=\dim (\Lambda U)^k-\dim U^k
    \end{split}
\end{equation*}
for $1\leq k\leq N-2$. Note that 
\begin{equation*}
    \begin{split}
        \alpha_1&=\dim A=\dim V^2-1, \\
        \alpha_2&=\dim B=\dim V^3-1, \\
        \alpha_k&=\dim V^{k+1}, \ \ \ \ {\rm for} \ \ 3\leq k\leq N-2.
    \end{split}
\end{equation*}
The following two important lemmas are immediate consequences of Propositions \ref{prop} and \ref{sig}. 

\begin{lemma}\label{lm1}
    For $1\leq k\leq N-2$, we have 
    $$\alpha_k=\alpha_{N-1-k}.$$

\end{lemma}

\begin{proof}
    According to Proposition \ref{sig}\,(3), we derive that 
    \begin{equation*}
        \begin{split}
            \alpha_k&= \dim \widetilde{U}^{k+1} \\
            &=\dim (V^{k+1})^\ast-\dim(\ker \Sigma)_{k+1} \\
            &= \dim \widetilde{U}^{N-k} \\
            &= \alpha_{N-1-k}
        \end{split}
    \end{equation*}
    for $1\leq k\leq N-2$. 
    
\end{proof}

\begin{lemma}\label{lm2}
    We have 
    \begin{equation*}
        \begin{split}
            &\beta_k=0,  \ \  \textit{ for } \ 1\leq k\leq N-4, \\
            &\beta_{N-3}\leq 1, \ \ \beta_{N-2}\leq 1.
        \end{split}
    \end{equation*}
    In particular, we have 
    \begin{enumerate}
        \item $\alpha_1\leq 1$; 

        \item Every nonzero element $u\in U^{even}$ satisfies $|u|\geq \frac{N-3}{2}$; 

        \item If $N$ is odd, then every nonzero element $u\in U^{even}$ also satisfies $|u|\leq \frac{N+1}{2}$.
    \end{enumerate}
    
\end{lemma}

\begin{proof}
    According to Proposition \ref{prop}, the map $L_{-1}\Phi|_{(\Lambda U)^{\leq N-2}}$ is injective. Thus for $1\leq k\leq N-4$, we have
    $$\dim (\Lambda U)^k \leq \dim (V^{N-k})^\ast= \dim U^{N-k-1}=\dim U^k,$$
    $$\dim (\Lambda U)^{N-3} \leq \dim (V^{3})^\ast= \dim U^{2}+1=\dim U^{N-3}+1,$$
    $$\dim (\Lambda U)^{N-2} \leq \dim (V^{2})^\ast= \dim U^{1}+1=\dim U^{N-2}+1.$$
    It follows that for $1\leq k\leq N-4$, 
    \begin{equation*}
        \begin{split}
            \beta_k&=\dim (\Lambda U)^k-\dim U^k\leq 0, \\
            \beta_{N-3}&=\dim (\Lambda U)^{N-3}-\dim U^{N-3}\leq 1, \\
            \beta_{N-2}&=\dim (\Lambda U)^{N-2}-\dim U^{N-2}\leq 1. 
        \end{split}
    \end{equation*}

    In particular, suppose that $\alpha_1\geq 2$. Then there exist two linearly independent elements $u_1,u_2\in U^1$. Thus $u_1u_2$ is a nonzero element in $(\Lambda^{\geq2}U)^2$, which contradicts the fact that $\beta_2=0$ for $N\geq 6$. 

    For the nonzero element $u\in U^{even}$, since $0\not= u^2\in (\Lambda^{\geq2}U)^{2|u|}$, it follows that 
    $$2|u|\geq N-3.$$
    Furthermore, if $N$ is odd, then $\alpha_{N-1-|u|}=\alpha_{|u|}\geq1$. Thus there exists a nonzero element in $U^{even}$ of degree $N-1-|u|$. Similarly, we have 
    $$2(N-1-|u|)\geq N-3,$$
    which implies $2|u|\leq N+1$. 
    
\end{proof}

\begin{remark}\label{rm}
    Lemma \ref{lm2}\,(3) asserts that if $N$ is odd, the only possible even degrees of nonzero elements in $U^{even}$ are 
    \begin{equation*}
        \left\{  \begin{array}{ll}
        \frac{N-1}{2} ,  &  \ N\equiv 1 \  ({\rm mod}\ 4),\\
        \frac{N-3}{2},  \frac{N+1}{2},   &  \ N\equiv 3 \  ({\rm mod}\ 4). 
        \end{array}   \right.
    \end{equation*}
    
\end{remark}

The following observation will be used repeatedly. 

\begin{lemma}\label{q}
    For any choice of generators for $V$, it is impossible that none of the differentials of generators contains $q$. In particular, it is impossible that all the generators other than $q,y$ have even degrees. 
    
\end{lemma}

\begin{proof}
    The first statement is obvious, since $q^\ast$ is a non-exact cycle of degree $3$ and it contradicts (\ref{td}). 

    Suppose that all the generators other than $q,y$ have even degrees. Then induction on generator degrees shows that these even generators are closed: their differentials must have the form $qF$ for a polynomial $F$, and thus $0=(dq)F$, which forces $F=0$. Finally, we can conclude that $dy$ does not contain $q$ by similar arguments, which contradicts the first statement. 
    
\end{proof}

Recall that Lemma \ref{lm2}\,(1) asserts $\alpha_1\leq 1$. We now rule out $\alpha_1=1$. 

\begin{lemma}\label{1}
    It is impossible that $\alpha_1=1$.
    
\end{lemma}

\begin{proof}
    Suppose that $\alpha_1=1$, and denote $U^1=\overline{A}=\Q \overline{a}$. By Lemma \ref{lm2}, we have $\beta_k=0$ for $3\leq k\leq N-4$, which implies that there exists no nonzero element in $U^{k-1}$, since its product with $\overline{a}$ has degree $k$. By Lemma \ref{lm1}, it follows that 
    $$\alpha_1=\alpha_{N-2}=1,\ \ \alpha_k=\alpha_{N-1-k}=0$$ 
    for $2\leq k\leq N-5$. Note that $4\leq N-1-k\leq N-3$ for $2\leq k\leq N-5$. Thus when $N\geq 8$, these two intervals cover the degrees from $2$ to $N-3$, which implies 
    $$U=U^1\oplus U^{N-2}=\Q \overline{a}\oplus \Q \overline{w}$$
    for nonzero $\overline{w}\in U^{N-2}=\overline{V^{N-1}}$. 

    Since $N\geq 6$, there are two remaining cases: 
    \begin{enumerate}
        \item When $N=7$, an additional generator might appear in degree $3$. There exists at most one additional generator, since its product with $\overline{a}$ has degree $4=N-3$. By Lemma \ref{q}, it is impossible. 

        \item When $N=6$, an additional generator might appear in degrees $2$ and $3$, and $\alpha_2=\alpha_3$. Suppose that there exist nonzero $u\in U^2$ and $u^\prime\in U^3$. Note that $u^2$ and $\overline{a}u^\prime$ are linearly independent elements of degree $4=N-2$, which contradicts $\beta_{N-2}\leq 1$. 
    \end{enumerate}

    We have concluded that there exist no other generators except $\overline{a},\overline{w}$ in $U$. It remains to consider the generators $\{x,a, q,w,y\}$ for $V$. If $N $ is odd, $x,a,w$ have even degrees, which leads to a contradiction by Lemma \ref{q}. If $N$ is even, by considering degrees and the differential $D$, we may assume  
    $$dq=Q,\ \ dw=F(x,a),\ \ dy=aw+qG(x,a),  $$
    where $Q,F,G$ are polynomials with the variables $x,a$. Note that 
    $$0=d^2y=aF+QG .$$ 
    Since $dq=Q \equiv x^2 \ ({\rm mod }\  (a))$, we conclude that 
    $$F=QR, \ \ G=-aR$$
    for some polynomial $R$. By the change of basis $ w-qR\mapsto w^\prime $, we can replace the original differential by 
    $$dq=Q,\ \ dw^\prime=0,\ \ dy=aw^\prime, $$
    which also leads to a contradiction by Lemma \ref{q}.

\end{proof}

\subsection{Possible degrees for the case $\delta=0$}\label{sub4}

\ 

As we can see from Lemmas \ref{lm1} and \ref{lm2}, the possible degrees for generators of $U$ are highly constrained. Recall that $\alpha_1=\alpha_{N-2}=0$ by Lemma \ref{1}, and $U^{even}$ is nontrivial by Lemma \ref{q}. Denote by $2m$ the least even degree for generators of $U^{even}$. Then according to Lemma \ref{lm2}\,(2), we have 
$$4m\geq N-3.$$ 

Recall that $N=n+1\geq6$. We first consider the case when $N$ is even. 

\begin{proposition}\label{even}
    If $N$ is even, it is impossible. 
    
\end{proposition}

\begin{proof}
    Recall that $$\alpha_k=\alpha_{N-1-k}$$ for $1\leq k\leq N-2$, and $k, N-1-k$ have different parity for $N$ even. Thus the odd degrees for generators of $U^{odd}$ are at most $N-1-2m$. 

    First suppose that $4m\geq N+2$. Then $2(N-1-2m)\leq N-4$. By Lemma \ref{lm2}, it follows that there cannot be two linearly independent elements in $U^{odd}$. Therefore, we conclude that there exists precisely one generator in $U^{odd}=U^{N-1-2m}$ and one generator in $U^{even}=U^{2m}$. Note that $N-1-2m<2m$. Denote the corresponding generators in $V$ by $l,h$ with 
    $$l\in V^{N-2m},\ \ h\in V^{2m+1}, \ \ |l|<|h|.$$
    Then the generators for $V$ are $\{x,q,l,h,y\}$. By considering degrees and $d^2=0$, we may assume 
    $$dq=x^2,\ \ dl=0,\ \ dh=F(x,l),\ \ dy=hG(x,l)+qH(x,l),$$
    where $F,G,H$ are polynomials with the variables $x,l$. Note that 
    $$0=d^2y=FG+x^2H.$$
    For degree reasons, we can set 
    $$G(x,l)=\lambda l+\mu x^{\frac{N}{2}-m}$$ 
    for some $\lambda,\mu \in \Q$. By Proposition \ref{sig}\,(3), $\Sigma(h^\ast)=\lambda l\neq 0$. Without loss of generality, we may assume $\lambda=1$, and it follows that $x\nmid G$. Thus, we can conclude that 
    $$F=x^2S, \ \ H=-SG$$
    for some polynomial $S$. By the change of basis $l+\mu x^{\frac{N}{2}-m}\mapsto l^\prime$ and $h-qS \mapsto h^\prime$, we can replace the original differential by 
    $$dq=x^2,\ \ dl^\prime=dh^\prime=0,\ \ dy=h^\prime l^\prime ,$$
    which contradicts Lemma \ref{q}. 

    It remains to consider the case $N-3\leq 4m\leq N+1$. If $N\equiv 0\ ({\rm mod}\ 4)$, we have $4m=N$, and the odd degrees for generators of $U^{odd}$ are at most $2m-1$. It follows that $U^{odd}=U^{2m-1}$, since every lower odd-degree generator admits a product of degree $(2m-3)+(2m-1)\leq N-4$. Thus the generators for $V$ are $\{x,q,l_i,h_i,y\}$ with 
    $$l_i\in V^{2m},\ \ h_i\in V^{2m+1}.$$ 
    By similar arguments as above, we may assume 
    $$dq=x^2, \ \ dl_i=0,\ \ dh_i=x\sum_j\lambda_{ij}l_j,\ \ dy=\sum_i h_il_i + qH(x,l_i) $$
    for some $\lambda_{ij}\in Q$, after a suitable change of basis. Note that 
    $$0=d^2y=x \sum_{i,j} \lambda_{ij}l_il_j +x^2H. $$
    Then we conclude that 
    $$\lambda_{ij}+\lambda_{ji}=0,\ \ H=0, $$
    which also contradicts Lemma \ref{q}. 

    If $N\equiv 2\ ({\rm mod}\ 4)$, we have $4m=N-2$, and the odd degrees for generators of $U^{odd}$ are at most $2m+1$. The squares of elements in $U^{2m}$ have degree $4m=N-2$, which implies $\alpha_{2m}=\alpha_{2m+1}=1$ by Lemma \ref{lm2}. We denote the corresponding generators in $V$ by $l,h$ with 
    $$l\in V^{2m+1},\ \ h\in V^{2m+2}. $$
    Thus, there exists no nonzero element in $U^{2m-1}$, since its product with $\overline{h}$ has degree $4m=N-2$, and is linearly independent from $\overline{l}^2$. Furthermore, there exists no nonzero element in $U^{odd}$ of degree $\leq 2m-3$, since its product with $\overline{h}$ has degree $\leq 4m-2=N-4$. Thus the generators for $V$ are $\{x,q,l,h,y\}$. By similar arguments as above, we may assume 
    $$dq=x^2,\ \ dl=0,\ \ dh=\lambda_1 xl+\lambda_2 qx^m, $$
    $$dy=lh+\mu_1 l x^{m+1}+\mu_2 qx^{2m} +\mu_3 q hx^{m-1} $$
    for some $\lambda_1,\lambda_2,\mu_1,\mu_2,\mu_3  \in \Q$ after a suitable change of basis. Note that 
    $$0=d^2h=\mu x^{m+2}, $$
    $$0=d^2y=\mu_2 x^{2m+2}+\mu_3 (hx^{m+1} - \lambda_1 qlx^m). $$
    It follows that $\lambda_2=\mu_2=\mu_3=0$, which also contradicts Lemma \ref{q}. This completes the proof of Proposition \ref{even}.

\end{proof}

We now consider the other case when $N$ is odd.

\begin{proposition}\label{odd}
    If $N$ is odd, then we have 
    $$N\equiv 3\ ({\rm mod}\ 4). $$
    Moreover, the generators for $V$ are $\{ x,q,a,z,w,y \}$ with degrees 
    $$|x|=2,\  |q|=3,\  |a|=\frac{N-1}{2},\  |z|=\frac{N+1}{2},\  |w|=\frac{N+3}{2}, \  |y|=N. $$
    
\end{proposition}

\begin{proof}
    Recall that $$\alpha_k=\alpha_{N-1-k}$$ for $1\leq k\leq N-2$, and $k, N-1-k$ have the same parity for $N$ odd. According to Remark \ref{rm}, we can show that 
    \begin{equation*}
        N= \left\{  \begin{array}{ll}
        4m+1,  &  \ N\equiv 1 \  ({\rm mod}\ 4),\\
        4m+3,   &  \ N\equiv 3 \  ({\rm mod}\ 4), 
        \end{array}   \right.
    \end{equation*}
    where $2m$ is the least even degree for generators of $U^{even}$. 

    If $N=4m+1$, then we have $U^{even}=U^{2m}$. We can conclude that there exists no nonzero element in $U^{odd}$ of degree $\leq 2m-3$, since its product with elements in $U^{2m}$ has degree $\leq4m-3 =N-4$. By Lemma \ref{lm1}, it follows that $U^{odd}=U^{2m-1}\oplus U^{2m+1}$ with $\alpha_{2m-1}=\alpha_{2m+1}$. Since 
    the products of an element in $U^{2m-1}$ and an element in $U^{2m}$ have degree $4m-1=N-2$, we can derive that 
    \begin{equation}\label{999}
        \alpha_{2m-1}\alpha_{2m}\leq 1.
    \end{equation}
    On the other hand, suppose that $\alpha_{2m}=1$. Denote by $z$ the generator of $\widetilde{U^{2m}}\subset V^{2m+1}$. Thus by Proposition \ref{sig}\,(3), $\Sigma(z^\ast)=\lambda z$ for some nonzero $\lambda\in \Q$. However, since $z^2=0$ and $\widetilde{U^{2m}}=\Q z$, the quadratic part of $dy$ cannot contain $z$, which implies $\Sigma(z^\ast)=0$, a contradiction. It follows that $\alpha_{2m}\geq 2$, and thus the inequality (\ref{999}) gives $\alpha_{2m-1}=\alpha_{2m+1}=0$. Similarly, by considering the differentials, this is impossible by Lemma \ref{q}. 

    If $N=4m+3$, then we have $U^{even}=U^{2m}\oplus U^{2m+2}$ with $\alpha_{2m}=\alpha_{2m+2}$. We can conclude that there exists no nonzero element in $U^{odd}$ of degree $\leq 2m-1$, since its product with elements in $U^{2m}$ has degree $\leq 4m-1 =N-4$. By Lemma \ref{lm1}, it follows that $U^{odd}= U^{2m+1}$. Since products of elements in $U^{2m}$ have degree $4m=N-3$, it follows that $\alpha_{2m}=\alpha_{2m+2}=1$. Denote the corresponding generators in $V$ by $a,w$ with
    $$a\in V^{2m+1},\ \ w\in V^{2m+3}.$$
    Similarly we have $\alpha_{2m+1}\leq 1$ by considering its product with elements in $U^{2m}$. 

    Finally, suppose that $\alpha_{2m+1}=0$. Then the generators for $V$ are $\{x,q,a,w,y\}$, and we may assume 
    $$dq=x^2, \ \ da=0, \ \ dw= \lambda qa, \ \ dy=aw+\mu qax^m+\gamma qwx^{m-1}$$
    for some $\lambda,\mu,\gamma\in\Q$, after a suitable change of basis. Note that 
    \begin{equation*}
        \begin{split}
            &0=d^2w=\lambda x^2a \ \ \implies \ \ \lambda=0, \\
            &0=d^2y=\mu ax^{m+2} +\gamma wx^{m+1} \ \ \implies \ \ \mu=\gamma=0, 
        \end{split}
    \end{equation*}
    which contradicts Lemma \ref{q}. Thus we have $\alpha_{2m+1}=1$, and denote the corresponding generator in $V$ by $z$ with $z\in V^{2m+2}$. This completes the proof of Proposition \ref{odd}.

\end{proof}

\subsection{Realization when $n\geq 5$}\label{sub5}

\ 

Propositions \ref{even} and \ref{odd} show that the only possible realization for $K(\Q,n)\times K(\Q,n+2)$ with $n\geq 5$ is 
\begin{equation}\label{model}
    \big( \Lambda(x_2,q_3,a_{2m+1},z_{2m+2},w_{2m+3},y_{4m+3}), d \big)
\end{equation}
when $n=4m+2$. After a suitable change of basis to kill all the powers of $x$, we may assume 
$$dx=0,\ \ dq=x^2,\ \ da=0,\ \ dz=\lambda_1 xa+\lambda_2 x^m q, \ \ dw=\mu_1 xz+\mu_2 qa,$$
$$dy= \gamma_1 x^{m+1}z +\gamma_2 z^2 +\gamma_3 x^m qa  +\gamma_4 x^{m-1} qw +aw$$
with $\lambda_1,\lambda_2,\mu_1,\mu_2,\gamma_1,\gamma_2,\gamma_3,\gamma_4\in \Q$. From 
$$0=d^2z=\lambda_2 x^{m+2},$$
we deduce that $\lambda_2=0$. The discussion is divided into the following two cases. 

\noindent {\bf Case 1:} $\lambda_1=0$. Then we have $dz=0$, and 
$$0=d^2w=\mu_2 x^2a$$
implies $\mu_2=0$. Then Case 1 can be divided into the following two subcases. 

\noindent {\bf Subcase 1.1:} $\mu_1=0$. Then we have $dw=0$, and 
$$0=d^2y=\gamma_3 x^{m+2}a+\gamma_4 x^{m+1}w$$
implies $\gamma_3=\gamma_4=0$. This is impossible by Lemma \ref{q}. 

\noindent {\bf Subcase 1.2:} $\mu_1\neq0$. Without loss of generality, we may assume $\mu_1=1$, i.e., $dw=xz$. Then we have
$$0=d^2y=\gamma_3 x^{m+2}a+\gamma_4 (x^{m+1}w -x^mqz)-axz.$$
This is impossible because the term $-axz$ cannot be canceled. 

\noindent {\bf Case 2:} $\lambda_1\neq 0$. Without loss of generality, we may assume $\lambda_1=1$, i.e., $dz=xa$. Then we have 
$$0=d^2w=\mu_1 x^2a+\mu_2 x^2a,$$
which implies $\mu_1=-\mu_2$. Then Case 2 can be divided into the following two subcases. 

\noindent {\bf Subcase 2.1:} $\mu_1=\mu_2=0$. Then we have $dw=0$. From 
$$0=d^2y=\gamma_1 x^{m+2}a + 2\gamma_2 xaz +\gamma_3 x^{m+2}a +\gamma_4 x^{m+1}w ,$$
we deduce that $\gamma_1=-\gamma_3$, $\gamma_2=\gamma_4=0$. By the change of basis $y-\gamma_1x^{m-1}qz$, we can replace the differential of $y$ by $dy=aw$, which contradicts Lemma \ref{q}. 

\noindent {\bf Subcase 2.2:} $\mu_1=-\mu_2\neq0$. Without loss of generality, we may take $\mu_1=1$ and $\mu_2=-1$, i.e., $dw=xz-qa$. Then we have 
\begin{equation*}
    \begin{split}
        0=d^2y&=\gamma_1 x^{m+2}a+2\gamma_2 xaz+\gamma_3 x^{m+2}a+\gamma_4 (x^{m+1}w -x^mqz) -a(xz-qa) \\
        &= (\gamma_1+\gamma_3)x^{m+2}a+(2\gamma_2-1)xaz+\gamma_4(x^{m+1}w -x^mqz), 
    \end{split}
\end{equation*}
which implies 
$$\gamma_1=-\gamma_3, \ \ \gamma_2=\frac{1}{2}, \ \ \gamma_4=0 .$$
Therefore, we conclude that 
$dy=\gamma_1 x^m(xz-qa) +\frac{1}{2} z^2+aw.$
Replacing $2(y-x^mw)$ by $y$ gives the required differential 
\begin{equation}\label{diff}
    \begin{split}
        dx&=da=0,\ \ dq=x^2,\ \ dz=xa, \\
        dw&=xz-qa,\ \ dy=z^2+2aw. 
    \end{split}
\end{equation}

Recall that $n=4m+2$, and denote by $M_X$ the minimal Sullivan algebra (\ref{model}) with the differential (\ref{diff}). It remains to show that the space $X$ with minimal Sullivan model $M_X$ indeed satisfies
$$\B(X)\simeq_{\Q}K(\Q,4m+2)\times K(\Q,4m+4) .$$

\begin{proposition}\label{geq5}
    Let $m$ be a positive integer, and let $M_X$ be the minimal Sullivan algebra defined as above. Then we have 
    \begin{equation*} 
        H_k({\rm Der}'(M_X))=\left\{
        \begin{array}{ll}
        \Q y^\ast, &k=4m+3,\\
        \Q (y,x), &k=4m+1,\\
        0, &{\rm otherwise}. 
        \end{array}
        \right.
    \end{equation*}
    
\end{proposition}

\begin{proof}
    For simplicity, we assign a new weight ${\rm wt}(-)$ on $M_X$ by setting 
    $${\rm wt}(x)={\rm wt}(q)=0,\ \ {\rm wt}(a)={\rm wt}(z)={\rm wt}(w)=1,\ \ {\rm wt}(y)=2. $$ 
    Note that the differential $d$ preserves the weight and increases wordlength by exactly one. For degree reasons, every positive-degree derivation has weight $-2,-1,0$. Thus we have the decomposition of chain complexes
    $$(\Der^\prime(\Lambda V),D)= Der^\prime_{-2}\oplus Der^\prime_{-1} \oplus Der^\prime_0, $$
    where $Der_{i}$ is the subcomplex of $(\Der(\Lambda V),D)$ generated by derivations that increase the weight by $i$, and $Der_{i}^\prime$ denotes its positive truncation for $i=-2,-1,0$. It suffices to prove 
    $$H(Der_{-2}^\prime)=\Q y^\ast \oplus \Q (y,x),\ \ H(Der_{-1}^\prime)=H(Der_{0}^\prime)=0.$$

    For $Der_{-2}$, it is generated by 
    $$y^\ast,(y,x),\dots,(y,q),(y,xq),\dots.$$
    Note that $D(y,x^kq)=(y,x^{k+2})$ for $k\geq 0$, and $Dy^\ast=D(y,x)=0$. It follows that 
    $$H(Der_{-2}^\prime)=\Q y^\ast \oplus \Q (y,x).$$

    For $Der_{-1}$, denote 
    $$W_0=\Q x \oplus \Q q,\ \ W_1=\Q a\oplus \Q z\oplus \Q w,$$
    and then $V=W_0\oplus W_1\oplus \Q y$ is the weight decomposition of $V$. We set 
    $$Der_{-1}= (P\oplus Q,D) ,$$
    where $P$, $Q$ denote the derivation subspaces of $Der_{-1}$ 
    $$\Hom(W_1,\Lambda W_0),\ \ \Hom(y,W_1\otimes\Lambda W_0),$$
    respectively. Note that the differential $D$ satisfies 
    $$D\big(\theta,\xi\big)=\big(D_P \theta, D_Q\xi+ \varphi(\theta) \big),\ \ \ \varphi(\theta)=-(-1)^{|\theta|} \big(y, \theta(dy) \big) $$
    where $D_P,D_Q$ are the projections of $D$ onto $P,Q$, respectively. By direct computation, we have 
    $$\varphi(a^\ast)= 2(y,w),\ \ \varphi(z^\ast)=-2(y,z),\ \ \varphi(w^\ast)=-2(y,a), $$
    which implies
    $$\varphi: W_1^\ast  \otimes \Lambda W_0 \cong P \longrightarrow Q\cong (y,W_1)\otimes \Lambda W_0 $$ 
    is an isomorphism of $(\Lambda W_0,d)$-modules. The identity $D^2=0$ gives 
    $$D_Q \varphi+\varphi D_P=0.$$
    Thus $Der_{-1}$ is an acyclic mapping cone, which implies $H(Der_{-1}^\prime)=0$. 

    For $Der_{0}$, there are exactly six generators of positive-degree derivations: 
    \begin{equation*}\centerline{
    \begin{tabular}{c|c}
    degree & generators\\
    \hline
    3&$q^\ast$\\
    \hline
    2&$ x^\ast, (w,a)$ \\
    \hline
    1& $(q,x),(z,a),(w,z)$.\\
    \end{tabular}}\end{equation*}
    By direct computation, we have 
    \begin{equation*}
        \begin{split}
            Dq^\ast&= -(w,a), \\
            Dx^\ast&= -2(q,x)-(z,a)-(w,z), \\
            D(w,a)&= 0, \\
            D(q,x)&= -(w,xa), \\
            D(z,a)&= (w,xa)+2(y,az), \\
            D(w,z)&= (w,xa)-2(y,az).
        \end{split}
    \end{equation*}
    It follows that $H_3(Der_{0})=H_2(Der_{0})=H_1(Der_{0})=0$, which implies $H(Der_{0}^\prime)=0$. This completes the proof of Proposition \ref{geq5}.

\end{proof}

Propositions \ref{even}, \ref{odd} and \ref{geq5} provide the complete classification for the branch $n\geq 5$.

\section{The branch $n\leq 4$}

In this section, we treat the remaining branch $n\leq 4$. Recall that $(\Lambda V,d)$ is the minimal Sullivan model for $X$ with $V=V^{\geq 2} $ of finite type, and $\B (X)$ has the rational homotopy type of $K(\Q, n) \times K(\Q,n+2) $.

\subsection{The cases $n=1$ and $n=2$}

\ 

In the case $n=1$, since $\B (X)$ is simply-connected, it is impossible that $\B (X)$ has the rational homotopy type of $K(\Q, 1) \times K(\Q,3)$. 

In the case $n=2$, we have 
$$V^3=\Q y,\ \ V^{\geq 4}=0.$$
The derivations are displayed as follows. 
\begin{equation*}
    \centerline{
    \begin{tabular}{c|c}
    degree & derivations\\
    \hline
    3&$\Q y^\ast$\\
    \hline
    2&$ (V^2)^\ast$ \\
    \hline
    1& $(y,V^2)$.\\
    \end{tabular}}
\end{equation*}
Note that the differential of the derivations satisfies 
$$\Q y^\ast \xrightarrow{D_3} (V^2)^\ast  \xrightarrow{D_2} (y,V^2) \xrightarrow{D_1} 0. $$
Thus the Euler characteristic of $(\Der^\prime(\Lambda V),D)$ is $-1$. However, the Euler characteristic of $H(\Der^\prime(\Lambda V),D)$ is $-2$ by (\ref{td0}), a contradiction. Thus we obtain the following proposition. 
\begin{proposition}\label{n12}
    The classifying space $\B (X)$ of a simply-connected $\pi$-finite space $X$ with rational homology of finite type cannot have the rational homotopy type of $K(\Q,1) \times K(\Q,3)$ or $K(\Q,2) \times K(\Q,4)$. 
    
\end{proposition}

\subsection{The case $n=3$}

\ 

In this case we have 
$$V^4=\Q y,\ \ V^{\geq 5}=0,$$
and Lemma \ref{codim} implies that $\dim V^2=\dim V^3+1$. 
The derivations are displayed as follows. 
\begin{equation*}
    \centerline{
    \begin{tabular}{c|c}
    degree & derivations\\
    \hline
    4&$\Q y^\ast$\\
    \hline
    3&$(V^3)^\ast$\\
    \hline
    2&$(V^2)^\ast \oplus(y,V^2)$ \\
    \hline
    1& $(V^3,V^2)  \oplus(y,V^3)$.\\
    \end{tabular}}
\end{equation*}
Set $\dim V^2=s+1$ and $\dim V^3=s$. 

Recall that we have defined 
$$\delta=\dim (\Lambda^2V^2) -\dim (V^2A+dV^3)\in \{0,1\},$$
where $A\subset V^2$ has codimension $1$, and the differential of the derivations satisfies 
\begin{equation*}
    \begin{split}
        &D_3: (V^3)^\ast\xrightarrow{\cong} (y,A) \\
        &D_2: (V^2)^\ast \hookrightarrow (V^3,V^2)  \oplus(y,V^3), \ \ D_2(y,V^2)=0, \\
        &D_1: (V^3,V^2)\oplus (y,V^3) \twoheadrightarrow (y,V^2A+dV^3). 
    \end{split}
\end{equation*}
We claim that $dV^3\subset V^2A$, i.e., $\delta=1$. Otherwise, for a basis $\{w_1,\dots,w_s\}$ for $V^3$, suppose that $dw_1=x^2+ \sum_{i=1}^s v_ia_i$, where $\{a_1,\dots,a_s\}$ satisfies 
$$Dw_1^\ast=(y,a_1), \dots, Dw_s^\ast=(y,a_s). $$
It follows that 
$$dy=\sum_{i=1}^s w_i a_i $$
for degree reasons. However, the equation $d^2y=0$ gives $x^2a_1=0$, a contradiction. Thus we have $dV^3\subset V^2A$. 

It follows that 
\begin{equation*}
    \begin{split}
        &\dim \im D_1=\dim (y,V^2A) =\binom{s+2}{2} -1, \\
        &\dim \ker D_1= \dim \im D_2=s+1, \\
        &\dim \,((V^3,V^2)  \oplus(y,V^3))=s(s+1)+s,
    \end{split}
\end{equation*}
Then the equation 
$$\dim \im D_1+ \dim \ker D_1= \dim \,((V^3,V^2)  \oplus(y,V^3))$$
implies $s=2$. Then we have 
$$V^4=\Q y,\ \ V^3=\Q w_1\oplus \Q w_2,\ \ V^2=\Q a_1\oplus \Q a_2\oplus \Q x$$
with $Dw_1^\ast=(y,a_1)$, $Dw_2^\ast=(y,a_2)$ and $dy=w_1a_1+w_2a_2$. Then the equation $d^2y=0$ implies 
$$(dw_1)a_1+(dw_2)a_2=0. $$
Then we can conclude that 
$$dw_1=a_2l,\ \ dw_2=-a_1l$$
for some $l\in V^2$. If $l\in \Q a_1\oplus \Q a_2$, then $x^\ast$ is a non-exact cycle derivation, a contradiction. Thus $\{a_1,a_2,l\}$ forms a new basis for $V^2$. Rewrite the generators $\{a_2,a_1,l,w_1,w_2,y  \}$ of $V$ as 
$$\{a,b,x,z,w,y\},$$ 
respectively, with the differential 
\begin{equation*}
    \begin{split}
        &da=db=dx=0,\ \ dz=ax, \\
        &dw=-bx,\ \ dy=aw+bz. 
    \end{split}
\end{equation*}

On the other hand, by direct computation, we can show the above Sullivan model indeed realizes $K(\Q, 3) \times K(\Q,5)$. Thus we obtain the following proposition. 
\begin{proposition}\label{n3}
    Let $X$ be a simply-connected $\pi$-finite space with rational homology of finite type, and let $M_X$ be the Sullivan model for $X$. Then the classifying space $\B (X)$ has the rational homotopy type of $K(\Q, 3) \times K(\Q,5)$ if and only if 
    \begin{equation*}
            \begin{split}
                M_X\cong\ &\big( \Lambda(a_2,b_2,x_2,z_3,w_3,y_4) ,d \big), \\
                &da=db=dx=0,\ \ dz=ax, \\
                &dw=-bx, \ \ dy=aw+bz. 
            \end{split}
        \end{equation*}
    
\end{proposition}

\subsection{The case $n=4$}

\ 

In this case we have 
$$V^5=\Q y,\ \ V^{\geq 6}=0,$$
and Lemma \ref{codim} implies that $\dim V^2=\dim V^4+1$. 
The derivations are displayed as follows. 
\begin{equation*}
    \centerline{
    \begin{tabular}{c|c}
    degree & derivations\\
    \hline
    5&$\Q y^\ast$\\
    \hline
    4&$(V^4)^\ast$\\
    \hline
    3&$(V^3)^\ast \oplus(y,V^2)$ \\
    \hline
    2& $(V^2)^\ast \oplus (V^4,V^2)  \oplus(y,V^3)$\\
    \hline
    1& $(V^3,V^2)\oplus (V^4,V^3)\oplus (y,V^4 \oplus \Lambda^2V^2) $.
    \end{tabular}}
\end{equation*}
Set $\dim V^2=s+1$, $\dim V^3=t$ and $\dim V^4=s$. 

Recall that we have defined 
$$\delta=\dim (\Lambda^2V^2) -\dim (V^2A+dV^3)\in \{0,1\},$$
where $A\subset V^2$ has codimension $1$, and the differential of the derivations satisfies 
\begin{equation*}
    \begin{split}
        &D_3: (V^3)^\ast \hookrightarrow (V^4,V^2)  \oplus(y,V^3), \ \ D_3(y,V^2)=0, \\
        &D_2: (V^4,V^2)\oplus (y,V^3) \twoheadrightarrow (y,V^2A+dV^3), \ \ D_2(V^2)^\ast \subset \Der_1^\prime(\Lambda V).
    \end{split}
\end{equation*}
Note that 
\begin{equation*}
    \begin{split}
        &\dim \im D_2 \leq \dim (V^2A+dV^3)+\dim V^2=\binom{s+2}{2}-\delta+s+1, \\
        &\dim \ker D_2=\dim \im D_3=t, \\
        &\dim \im D_2+\dim \ker D_2=\dim \,\Der^\prime_2(\Lambda V) =(s+1)^2+t. 
    \end{split}
\end{equation*}
It follows that 
$$\frac{s(s-1)}{2}\leq 1-\delta\leq 1, $$
which implies $s=0,1,2$. As in Proposition \ref{sig}, we can define a map 
$$\Sigma_3: (V^3)^\ast\longrightarrow V^3$$
by setting $\Sigma_3(\theta)=(D\theta)(y)$. 

\begin{lemma}\label{s01}
    Suppose $s=0$ or $1$. Then the map $\Sigma_3$ is an isomorphism. Moreover, there exists a basis for $V^3$
    $$\{u_1,v_1,\dots,u_r,v_r\}$$
    satisfying $\Sigma_3 u_i^\ast =v_i$ and $\Sigma_3 v_i^\ast =-u_i$ for $1\leq i\leq r$, where $\{u_1^\ast,v_1^\ast,\dots,u_r^\ast,v_r^\ast\}$ is the dual basis for $(V^3)^\ast$. 
    
\end{lemma}

\begin{proof}
    By direct computation, we have 
    $$\theta_1(\Sigma_3\theta_2)=-\theta_2(\Sigma_3\theta_1),\ \ \theta_1,\theta_2\in (V^3)^\ast. $$
    Note that the restriction of the differential $D_2$ to $(V^4,V^2)$ is injective for $s=0,1$. The vanishing of homology in degree $2$ implies that the composition 
    $$(V^3)^\ast \xrightarrow{\ D_3 \ } (V^4,V^2)  \oplus(y,V^3) \xrightarrow{ {\rm projection}  } (y,V^3) $$
    is also injective, which implies $\Sigma_3$ is an injection, hence an isomorphism. 

    Moreover, $(\theta_1,\theta_2) \mapsto \theta_1(\Sigma_3\theta_2)$ defines a non-degenerate anti-symmetric bilinear map on $(V^3)^\ast$, which is a symplectic form. Thus, we can choose a symplectic basis for $(V^3)^\ast$ 
    $$\{u_1^\ast,v_1^\ast,\dots,u_r^\ast,v_r^\ast\}$$ 
    satisfying $u_i^\ast(\Sigma_3u_j^\ast)=v_i^\ast(\Sigma_3v_j^\ast)=0$ and $v_i^\ast(\Sigma_3u_j^\ast)=-u_i^\ast(\Sigma_3v_j^\ast)=\delta_{ij}$ for $1\leq i,j\leq r$. Their dual basis 
    $$\{u_1,v_1,\dots,u_r,v_r\}$$ 
    forms a basis for $V^3$. 

\end{proof}

For degree reasons, we can write 
$$dy=R+S+T,$$
where $R\in \Lambda^3V^2$, $S\in V^2V^4$ and $T\in \Lambda^2V^3$. Choose a basis $\{w_1,\dots,w_s\}$ for $V^4$, and then we have $\{a_1,\dots,a_s,x\}$ as a basis for $V^2$, where $\{a_1,\dots,a_s\}$ satisfies 
$$Dw_1^\ast=-(y,a_1), \dots, Dw_s^\ast=-(y,a_s). $$
Then we can deduce that 
$$S= \sum_{i=1}^s a_iw_i .$$
And we also have 
$$T=\sum_{j=1}^r u_jv_j$$ 
for $s=0,1$. The following discussion is divided into three cases for $s=0,1,2$. 

If $s=0$, then we have $V^4=0$, $ V^2=\Q x$. For degree reasons, we can set $R=\lambda x^3$ for some $\lambda\in\Q$. From the equation $d^2y=0$, we can deduce $du_j=dv_j=0$ for $1\leq j\leq r$. Suppose that $\lambda=0$. Then we have $dy=\sum_{j=1}^r u_jv_j$, which implies that $x^\ast$ is a non-exact cycle, a contradiction. Thus $\lambda\neq0$, and we may assume $\lambda=1$. This gives the generators of $V$
$$\{x,u_1,v_1,\dots,u_r,v_r,y\}$$
for $r\geq 0$, with the differential 
$$dx=du_{i}=dv_{i}=0, \ \ dy=x^3+\sum_{j=1}^r  u_jv_j, $$
and we can show that such a Sullivan model indeed realizes $K(\Q, 4) \times K(\Q,6)$ by direct computation. 

If $s=1$, then we have $V^4=\Q w$, $ V^2=\Q a\oplus\Q x$. Since $R\in \Lambda^3V^2$, we can set
$$R=\lambda_1 x^3+\lambda_2x^2a+\lambda_3 xa^2+\lambda_4 a^3$$
for some $\lambda_1,\lambda_2,\lambda_3,\lambda_4 \in \Q$. By the change of basis $w+\lambda_2 x^2+\lambda_3 xa+\lambda_4 a^2 \mapsto w$, we can replace the differential $dy$ by 
$$dy=\lambda_1 x^3+aw + \sum_{j=1}^r u_jv_j.$$ 
The equation $d^2y=0$ gives 
$$a(dw)+ \sum_{j=1}^r (du_j)v_j-\sum_{j=1}^r u_j(dv_j)=0.$$
Note that $dw\in V^2 V^3$ and $du_j,dv_j\in \Lambda^2V^2$ for degree reasons. Then we conclude that 
$$dw=\sum_{j=1}^r (u_jl_{v_j}-l_{u_j}v_j),$$
where $du_j=al_{u_j}$, $dv_j=al_{v_j}$. By direct computation, the following derivations 
$$(y,a^2),\ (y,ax),\ (y,x^2),\ (u_j,a)-(w,v_j), \ (v_j,a)+(w,u_j)$$
for $1\leq j\leq r$ are linearly independent cycles of degree $1$, which implies $\dim \ker D_1\geq 2r+3$. On the other hand, we have concluded that $\dim\im D_2=(s+1)^2=4$. The vanishing of homology in degree $1$ gives  
$$4=\dim\im D_2=\dim \ker D_1\geq 2r+3,$$
implying $r=0$. And we can similarly conclude that $\lambda_1\neq 0$; otherwise $x^\ast$ would be a non-exact cycle. Without loss of generality, we may assume $\lambda_1=1$. This gives the generators of $V$
$$\{a,x,w,y\}$$
with the differential 
$$da=dx=dw=0, \ \ dy=x^3+aw, $$
and we can show that such a Sullivan model indeed realizes $K(\Q, 4) \times K(\Q,6)$ by direct computation. 

If $s=2$, then we have 
$$V^4=\Q w_1\oplus \Q w_2,\ \ V^2=\Q a_1\oplus \Q a_2\oplus\Q x.$$
Note that the kernel of the differential 
$$D_2: (V^4,V^2)\longrightarrow (y,\Lambda^2V^2)$$
is generated by $(w_1,a_2)-(w_2,a_1)$, which implies that there exists a derivation $z^\ast\in (V^3)^\ast$ such that $Dz^\ast=(w_1,a_2)-(w_2,a_1)$. Denote the image of the composition 
\begin{equation}\label{ty}
    (V^3)^\ast \xrightarrow{\ D_3 \ } (V^4,V^2)  \oplus(y,V^3) \xrightarrow{ {\rm projection}  } (y,V^3)
\end{equation}
by $(y,C)$ for a subspace $C\subset V^3$. For the map $\Sigma_3$, we can obtain the following lemma. 
\begin{lemma}\label{s2}
    Suppose $s=2$. Then the map $\Sigma_3$ satisfies 
    $$\ker \Sigma_3=\Q z^\ast, \ \  \im \Sigma_3=C. $$
    Moreover, we have $V^3=\Q z\oplus C$, and there exists a basis for $V^3$ 
    $$\{z,u_1,v_1,\dots,u_r,v_r\}$$
    satisfying $\Sigma_3 u_i^\ast =v_i$ and $\Sigma_3 v_i^\ast =-u_i$ for $1\leq i\leq r$, where $\{z^\ast,u_1^\ast,v_1^\ast,\dots,u_r^\ast,v_r^\ast\}$ is the dual basis for $(V^3)^\ast$. 
    
\end{lemma}

\begin{proof}
    By direct computation, we have 
    $$\theta_1(\Sigma_3\theta_2)=-\theta_2(\Sigma_3\theta_1),\ \ \theta_1,\theta_2\in (V^3)^\ast.$$
    Note that the composition (\ref{ty}) sends a derivation $\theta$ exactly to $(y,\Sigma_3\theta)$. Since $Dz^\ast=(w_1,a_2)-(w_2,a_1)$, we can conclude that 
    $$\ker \Sigma_3=\Q z^\ast, \ \  \im \Sigma_3=C, $$
    and $C=\ker z^\ast$. 
    
    We can decompose $V^3=\Q z\oplus C$, and thus $(V^3)^\ast=\Q z^\ast \oplus C^\ast$. Then we have the restriction map 
    $$\Sigma_3|_{C^\ast}: C^\ast \longrightarrow C.$$
    Similarly, $(\theta_1,\theta_2) \mapsto \theta_1(\Sigma_3\theta_2)$ defines a non-degenerate anti-symmetric bilinear map on $C^\ast$, and we can choose a symplectic basis for $C^\ast$ 
    $$\{u_1^\ast,v_1^\ast,\dots,u_r^\ast,v_r^\ast\}$$ 
    satisfying $u_i^\ast(\Sigma_3u_j^\ast)=v_i^\ast(\Sigma_3v_j^\ast)=0$ and $v_i^\ast(\Sigma_3u_j^\ast)=-u_i^\ast(\Sigma_3v_j^\ast)=\delta_{ij}$ for $1\leq i,j\leq r$. Their dual basis 
    $$\{u_1,v_1,\dots,u_r,v_r\}$$ 
    forms a basis for $C$. Finally, the proof is completed by adding $z^\ast$, $z$ to the bases for $C^\ast$, $C$, respectively. 
    
\end{proof}
Lemma \ref{s2} implies 
$$T=\sum_{j=1}^r u_jv_j$$ 
for $s=2$, and thus we have 
$$dy=R+a_1w_1+a_2w_2+\sum_{j=1}^r u_jv_j. $$
Since $R\in \Lambda^3V^2$, we can set 
$$R=\lambda x^3 +a_1F_1+a_2F_2  $$
for some $F_1,F_2\in \Lambda^2V^2$. By the change of basis $w_1+F_1\mapsto w_1$ and $w_2+F_2\mapsto w_2$, we can assume $F_1=F_2=0$ without loss of generality. Write the differential as 
$$dw_1=l_1z+G_1,\ \ dw_2=l_2z+G_2$$
for some $l_1,l_2\in V^2$ and $G_1,G_2\in V^2C$. Then the equation $d^2y=0$ gives 
$$(a_1l_1+a_2l_2)z+(a_1G_1+a_2G_2+H)=0, $$
where $H=\sum_{j=1}^r ((du_j)v_j-u_j(dv_j))$. Note that 
$$(a_1l_1+a_2l_2)z\in \Lambda^2V^2\cdot z,\ \ a_1G_1+a_2G_2+H\in \Lambda^2V^2\cdot C,$$
and it follows that 
$$a_1l_1+a_2l_2=0,\ \ a_1G_1+a_2G_2+H=0. $$
Note that it is impossible that $l_1=l_2=0$; otherwise $z^\ast$ would be a non-exact cycle. Then from $a_1l_1+a_2l_2=0$, we may assume $l_1=a_2$, $l_2=-a_1$ without loss of generality. By direct computation, we have 
\begin{equation*}
    \begin{split}
        D (u_j,a_1)&=(y,a_1v_j)+(w_1,u_j^\ast(a_1G_1)) +(w_2,u_j^\ast(a_1G_2)), \\
        D(w_1,v_j)&=(y,a_1v_j)+ (w_1,dv_j), \\
        D(z,u_j^\ast(G_2) )&=(w_1,u_j^\ast(a_2G_2))-(w_2,u_j^\ast(a_1G_2)). \\
    \end{split}
\end{equation*}
for $1\leq j\leq r$. Since $dv_j=-u_j^\ast (H)$ and $a_1G_1+a_2G_2+H=0 $, we derive that 
$$(u_j,a_1)-(w_1,v_j)+(z,u_j^\ast(G_2) )$$
is a cycle derivation of degree $1$. Similarly, we can conclude that the following derivations 
$$(y,x^2),\ (y,xa_1),\ (y,xa_2),\ (y,a_1^2),\ (y,a_1a_2),\ (y,a_2^2),$$
$$(u_j,a_1)-(w_1,v_j)+(z,u_j^\ast(G_2) ), \ (u_j,a_2)-(w_2,v_j)-(z,u_j^\ast(G_1) ), $$
$$(v_j,a_1)+(w_1,u_j)+(z,v_j^\ast(G_2) ) , \ (v_j,a_2)+(w_2,u_j)-(z,v_j^\ast(G_1) )$$
for $1\leq j\leq r$ are linearly independent cycles of degree $1$, which implies $\dim \ker D_1\geq 4r+6$. On the other hand, we have concluded that $\dim\im D_2=(s+1)^2=9$. The vanishing of homology in degree $1$ gives  
$$9=\dim\im D_2=\dim \ker D_1\geq 4r+6,$$
implying $r=0$, i.e., $C=0$ and $V^3=\Q z$. Then we have 
$$0=d^2w_1=a_2dz,$$
implying $dz=0$. Since $z$ does not appear in $dy$, it follows that $(y,z)$ is a non-exact cycle, a contradiction. 

Combining the three cases $s=0,1,2$ above, we obtain the following proposition. 
\begin{proposition}\label{n4}
    Let $X$ be a simply-connected $\pi$-finite space with rational homology of finite type, and let $M_X$ be the Sullivan model for $X$. Then the classifying space $\B (X)$ has the rational homotopy type of $K(\Q, 4) \times K(\Q,6)$ if and only if 
    \begin{equation*}
        \begin{split}
            M_X\cong\ &\big( \Lambda(x_2,u_{1,3},v_{1,3},\dots,u_{r,3},v_{r,3},y_5) ,d \big), \ \  \textit{   for }  r\geq0, \\ 
            &dx=du_{i}=dv_{i}=0, \ \ dy=x^3+\sum_{j=1}^r  u_jv_j,
        \end{split}
    \end{equation*}
    or 
    \begin{equation*}
        \begin{split}
            M_X\cong\ &\big( \Lambda(a_2,x_2,w_4,y_5) ,d \big), \\
            &da=dx=dw=0, \ \ dy=x^3+aw.
        \end{split}
    \end{equation*}
        
\end{proposition}

Propositions \ref{n12}, \ref{n3} and \ref{n4} provide the complete classification for the branch $n\leq 4$. Finally, we can deduce our main theorem. 
\begin{proof}[Proof of Theorem \ref{fi}]
	It follows immediately from Propositions \ref{even}, \ref{odd}, \ref{geq5}, \ref{n12}, \ref{n3} and \ref{n4}. 
    
\end{proof}

\end{document}